\documentclass[11pt,letterpaper]{article}

\usepackage[margin=1in]{geometry}
\usepackage{amsmath,amssymb,amsthm,mathtools}
\usepackage{microtype}
\usepackage{xcolor}
\usepackage{hyperref}

\hypersetup{
  colorlinks=true,
  linkcolor=blue!45!black,
  citecolor=blue!45!black,
  urlcolor=blue!55!black,
  pdftitle={The Skew Hamming Set-Pair Problem via Characteristic-Two Algebra},
  pdfauthor={Guorong Gao and Run Zhao},
  pdfsubject={Extremal set theory, Hamming spaces, and algebraic rank methods}
}

\newtheorem{theorem}{Theorem}[section]
\newtheorem{lemma}[theorem]{Lemma}
\theoremstyle{definition}
\newtheorem{definition}[theorem]{Definition}
\theoremstyle{remark}
\newtheorem{remark}[theorem]{Remark}

\newcommand{\dist}{\operatorname{dist}}
\newcommand{\F}{\mathbb{F}}

\title{Tight bound for the skew Hamming set-pair problem}
\author{Guorong Gao\thanks{This research is supported by National Key R\&D Program of China (Grant No.~2023YFA1010202), National Natural Science Foundation of China (Grant No.~12401448), and Natural Science Foundation of Fujian Province (Grant No.~2024J08030). Email addresses: \href{mailto:grgao@fzu.edu.cn}{grgao@fzu.edu.cn}, \href{mailto:runzhao0628@163.com}{runzhao0628@163.com}.}, Run Zhao\\\\
\small School of Mathematics and Statistics, Fuzhou University, Fuzhou, Fujian, China}
\date{}

\begin{document}

\maketitle

\begin{abstract}
Let $X$ be an alphabet, let $t\geq 0$ and $n\geq t+1$, and let $((a_i,b_i))_{i=1}^{m}$ be an ordered family of word pairs in $X^n$ satisfying $\dist(a_i,b_i)\geq t+1$ for every $i$ and $\dist(a_i,b_j)\leq t$ whenever $i<j$. We prove the sharp bound $m\leq 2^{t+1}$, thereby resolving a problem posed by Alon, Jin, and Sudakov. Our proof uses a linear-algebraic method based on a characteristic-two algebra, which may be of independent interest.
\end{abstract}

\section{Introduction and historical background}

\subsection{Set-pair inequalities}

For a positive integer $r$, write $[r]=\{1,\ldots,r\}$.
A \emph{set-pair system} is a family
\[
  \mathcal{P}=\{(A_i,B_i):i\in[m]\}
\]
such that $A_i\cap B_i=\varnothing$ for every $i$.
Here $m=|\mathcal{P}|$ is the number of ordered pairs in the system.
In the classical two-sided setting, one also assumes that $A_i\cap B_j\neq\varnothing$ whenever $i\neq j$.
Bollob\'as proved that every such system satisfies
\begin{equation}
  \sum_{i=1}^{m}
  \binom{|A_i|+|B_i|}{|A_i|}^{-1}
  \leq 1.
  \label{eq:bollobas}
\end{equation}
In particular, if $|A_i|=a$ and $|B_i|=b$ for all $i$, then
$m\leq \binom{a+b}{a}$ \cite{Bollobas1965}.
This inequality is a foundational instance of the set-pair method; see, for example, \cite{FranklTokushige2018,Tuza1994}.

The \emph{skew} condition requires $A_i\cap B_j\neq\varnothing$ only for $i<j$.
Lov\'asz's exterior-algebra argument shows that the same sharp uniform bound
$\binom{a+b}{a}$ holds under this one-sided condition; Kalai later obtained the result by related algebraic methods \cite{Lovasz1977,Kalai1984}.
F\"uredi proved the following common-intersection extension: if
\[
  |A_i|=a,\qquad |B_i|=b,\qquad |A_i\cap B_i|\leq c,
\]
and $|A_i\cap B_j|>c$ whenever $i<j$, then
\begin{equation}
  m\leq \binom{a+b-2c}{a-c}.
  \label{eq:furedi}
\end{equation}
This bound is also sharp \cite{Furedi1984}.
Further developments include nonuniform, weighted, subspace, and multipartite variants \cite{HegedusFrankl2023,Tuza1987}.

\subsection{The Hamming-space problem}

Let $X$ be an arbitrary alphabet and let $X^n$ denote the set of words
$x=(x_1,\ldots,x_n)$ of length $n$ over $X$.
The \emph{Hamming distance} between $x,y\in X^n$ is
\[
  \dist(x,y)=\bigl|\{k\in[n]:x_k\neq y_k\}\bigr|.
\]
The radius-$t$ Hamming ball centered at $x$ is
$B_t(x)=\{y\in X^n:\dist(x,y)\leq t\}$.

Questions about Helly numbers of Hamming balls arose in online learning with set-valued feedback \cite{RamanSubediTewari2023}.
Alon, Jin, and Sudakov subsequently proved that, for $n>t$ and $|X|\geq 2$, the Helly number of radius-$t$ Hamming balls in $X^n$ is exactly $2^{t+1}$ \cite{AlonJinSudakov}.
Their argument also yields the following two-sided set-pair statement: if
\[
  \dist(a_i,b_i)\geq t+1
  \quad\text{and}\quad
  \dist(a_i,b_j)\leq t\quad(i\neq j),
\]
then $m\leq 2^{t+1}$.

They asked whether the same bound remains valid under the one-sided assumption
\begin{equation}
  \dist(a_i,b_j)\leq t\qquad(i<j),
  \label{eq:skew-condition}
\end{equation}
This is the skew Hamming set-pair problem considered here.
F\"uredi's inequality \eqref{eq:furedi} gives the previously known estimate
\begin{equation}
  m\leq \binom{2t+2}{t+1}.
  \label{eq:previous-bound}
\end{equation}
Indeed, associate with each word $x\in X^n$ the $n$-element set
\[
  \widehat{x}=\{(k,x_k):k\in[n]\}.
\]
Then
$|\widehat{x}\cap\widehat{y}|=n-\dist(x,y)$.
The diagonal assumption gives
$|\widehat{a_i}\cap\widehat{b_i}|\leq n-t-1$, while
\eqref{eq:skew-condition} gives
$|\widehat{a_i}\cap\widehat{b_j}|\geq n-t$ for $i<j$.
Substituting $a=b=n$ and $c=n-t-1$ into \eqref{eq:furedi} yields \eqref{eq:previous-bound}.

\subsection{Our contribution}

We answer this problem affirmatively by proving the sharp bound.
The main ingredient is a low-dimensional bilinear representation whose zero pattern detects the threshold relation $\dist(x,y)\leq t$ exactly.
The representation is built in the commutative algebra
\[
  K[e_1,\ldots,e_d]/(e_1^2,\ldots,e_d^2),
\]
which has dimension $2^d$ over a field $K$ of characteristic two.
Characteristic two is used in two essential ways: every linear form in the $e_r$ squares to zero, and determinants agree with permanents.
Related algebraic ideas appear in determinantal sieving \cite{EibenKoanaWahlstrom2025}, while the final triangular-matrix argument is in the spirit of exterior-algebra proofs of skew set-pair inequalities \cite{Lovasz1977,HegedusFrankl2023}.

\section{Statement and sharpness}

\begin{definition}[Skew Hamming set-pair system]
\label{def:skew-hamming-system}
Let $t\in\mathbb{Z}_{\geq 0}$ and let $m,n\in\mathbb{Z}_{\geq 1}$ with $n\geq t+1$.
For an alphabet $X$, an ordered family
\[
  \mathcal{H}=((a_i,b_i))_{i=1}^{m},
  \qquad a_i,b_i\in X^n,
\]
is called a \emph{skew Hamming set-pair system at threshold $t$} if
\begin{align*}
  \dist(a_i,b_i)&\geq t+1 &&\text{for every }i\in[m],\\
  \dist(a_i,b_j)&\leq t   &&\text{whenever }1\leq i<j\leq m.
\end{align*}
The integer $m=|\mathcal{H}|$ is the number of ordered word pairs in the family and is the size parameter to be bounded.
\end{definition}

\begin{theorem}[Skew Hamming set-pair bound]
\label{thm:main}
Let $t\in\mathbb{Z}_{\geq 0}$ and let $m,n\in\mathbb{Z}_{\geq 1}$ with $n\geq t+1$.
Let $X$ be an alphabet, and let $a_i,b_i\in X^n$ for $i\in[m]$.
Assume that
\begin{align*}
  \dist(a_i,b_i)&\geq t+1 &&\text{for every }i\in[m],\\
  \dist(a_i,b_j)&\leq t   &&\text{whenever }1\leq i<j\leq m.
\end{align*}
Thus $\mathcal{H}=((a_i,b_i))_{i=1}^{m}$ is a skew Hamming set-pair system at threshold $t$ in the sense of Definition~\ref{def:skew-hamming-system}.
Then
\[
  m\leq 2^{t+1}.
\]
If $|X|\geq 2$, equality is attained already when $n=t+1$.
\end{theorem}

\begin{proof}[\textbf{Sharpness}]
Set $d=t+1$ and take $n=d$.
Choose two symbols in $X$, denoted by $0$ and $1$.
List the $2^d$ binary words of length $d$ as
$a_1,\ldots,a_{2^d}$, and let $b_i=\overline{a_i}$ be the coordinatewise complement of $a_i$.
Then $\dist(a_i,b_i)=d$ for every $i$.
Moreover, if $i\neq j$, then
\[
  \dist(a_i,b_j)
  =\dist(a_i,\overline{a_j})
  =d-\dist(a_i,a_j)
  \leq d-1=t.
\]
Thus the stronger two-sided cross-condition holds, and
$m=2^d=2^{t+1}$.
\end{proof}

The upper bound will follow from a bilinear kernel with exactly the required zero pattern.

\section{The characteristic-two algebraic kernel}

Set $d=t+1$.
Only finitely many symbols occur in the given family; let $\Sigma\subseteq X$ be their set.
Choose an infinite field $K_0$ of characteristic two and an injective coding map
\[
  c:\Sigma\longrightarrow K_0.
\]
Such a choice is always possible: for example, one may take the rational function field
$K_0=\F_2(s)$ and choose distinct elements for the finitely many symbols in $\Sigma$.

Let $z_{k,r}$, for $k\in[n]$ and $r\in[d]$, be algebraically independent over $K_0$, and work over the rational function field
\[
  K=K_0(z_{k,r}:k\in[n],\ r\in[d]).
\]
Consider the graded commutative $K$-algebra
\begin{equation}
  \mathcal{A}=K[e_1,\ldots,e_d]/(e_1^2,\ldots,e_d^2).
  \label{eq:algebra}
\end{equation}
For $S\subseteq[d]$, write
$e_S=\prod_{r\in S}e_r$, with $e_{\varnothing}=1$.
The squarefree monomials form a basis:
\begin{equation}
  \{e_S:S\subseteq[d]\}\text{ is a basis of }\mathcal{A},
  \qquad \dim_K\mathcal{A}=2^d.
  \label{eq:basis}
\end{equation}
For each coordinate $k\in[n]$, define
\begin{equation}
  u_k=\sum_{r=1}^{d}z_{k,r}e_r\in\mathcal{A}.
  \label{eq:uk}
\end{equation}

\begin{lemma}
\label{lem:square-zero}
For every $k\in[n]$, one has $u_k^2=0$.
\end{lemma}

\begin{proof}
Using commutativity, the relations $e_r^2=0$, and the identity $2=0$ in $K$, we obtain
\[
  u_k^2
  =\sum_{r=1}^{d}z_{k,r}^{2}e_r^2
   +2\sum_{1\leq r<s\leq d}z_{k,r}z_{k,s}e_re_s
  =0.
\]
\end{proof}

Encode a word $x\in\Sigma^n$ as
\begin{equation}
  P(x)=\prod_{k=1}^{n}\bigl(1+c(x_k)u_k\bigr)\in\mathcal{A}.
  \label{eq:encoding}
\end{equation}
Let $L:\mathcal{A}\to K$ be the linear map that extracts the coefficient of the top-degree monomial
$e_{[d]}=e_1\cdots e_d$.
Define a bilinear form $\beta:\mathcal{A}\times\mathcal{A}\to K$ by
\begin{equation}
  \beta(p,q)=L(pq).
  \label{eq:beta}
\end{equation}
In the squarefree basis,
\[
  \beta(e_S,e_T)=
  \begin{cases}
    1,&T=[d]\setminus S,\\
    0,&\text{otherwise}.
  \end{cases}
\]
Thus the matrix of $\beta$ is a permutation matrix.

\begin{lemma}[Exact zero pattern]
\label{lem:zero-pattern}
For all $a,b\in\Sigma^n$,
\[
  \beta(P(a),P(b))=0
  \quad\Longleftrightarrow\quad
  \dist(a,b)\leq d-1.
\]
Equivalently,
\[
  \beta(P(a),P(b))=
  \begin{cases}
    0,&\dist(a,b)\leq d-1,\\
    \neq 0,&\dist(a,b)\geq d.
  \end{cases}
\]
\end{lemma}

\begin{proof}
Let
\[
  D=\{k\in[n]:a_k\neq b_k\}.
\]
For each coordinate $k$, Lemma~\ref{lem:square-zero} gives
\begin{align*}
  \bigl(1+c(a_k)u_k\bigr)\bigl(1+c(b_k)u_k\bigr)
  &=1+\bigl(c(a_k)+c(b_k)\bigr)u_k+c(a_k)c(b_k)u_k^2\\
  &=1+\bigl(c(a_k)+c(b_k)\bigr)u_k.
\end{align*}
If $a_k=b_k$, the linear coefficient is $2c(a_k)=0$.
If $a_k\neq b_k$, injectivity of $c$ and characteristic two imply that
\[
  \alpha_k:=c(a_k)+c(b_k)\neq 0.
\]
By commutativity, we may group the factors coordinatewise.
Consequently,
\begin{equation}
  P(a)P(b)=\prod_{k\in D}(1+\alpha_k u_k).
  \label{eq:collapsed-product}
\end{equation}

If $|D|\leq d-1$, the right-hand side of \eqref{eq:collapsed-product} has degree at most $d-1$ in the variables $e_1,\ldots,e_d$.
Its $e_{[d]}$ coefficient is therefore zero.

Now suppose that $|D|\geq d$.
For $S=\{k_1<\cdots<k_d\}\subseteq D$, define
\[
  Z_S=(z_{k_i,r})_{i,r\in[d]}.
\]
The degree-$d$ part of \eqref{eq:collapsed-product} gives
\begin{equation}
  \beta(P(a),P(b))
  =\sum_{\substack{S\subseteq D\\|S|=d}}
    \left(\prod_{k\in S}\alpha_k\right)\det Z_S.
  \label{eq:determinant-sum}
\end{equation}
Indeed, the coefficient of $e_1\cdots e_d$ in
$\prod_{k\in S}u_k$ is the permanent of $Z_S$.
In characteristic two, every sign equals $1$, so the permanent and determinant coincide.

It remains to show that the sum in \eqref{eq:determinant-sum} does not cancel.
Fix $S_0=\{k_1<\cdots<k_d\}\subseteq D$.
The monomial
\[
  z_{k_1,1}z_{k_2,2}\cdots z_{k_d,d}
\]
occurs exactly once in $\det Z_{S_0}$, in the term indexed by the identity permutation, and its coefficient in \eqref{eq:determinant-sum} is
$\prod_{k\in S_0}\alpha_k\neq 0$.
A determinant indexed by a different row set uses a different collection of row variables and cannot contain this monomial.
Thus the right-hand side of \eqref{eq:determinant-sum} is a nonzero polynomial in
$K_0[z_{k,r}:k\in[n],\ r\in[d]]$, and hence is nonzero in its field of fractions $K$.
\end{proof}

\begin{remark}
\label{rem:indeterminates}
The algebraically independent variables serve only to prevent cancellation in
\eqref{eq:determinant-sum}; no probabilistic argument is involved.
The upper-bound proof requires only the resulting kernel over the rational function field $K$.
\end{remark}

\section{The triangular rank argument}

\begin{proof}[Proof of the upper bound in Theorem~\ref{thm:main}]
Set $d=t+1$, and construct $K$, $\mathcal{A}$, $P$, and $\beta$ as in Section~3.
Form the $m\times m$ matrix over $K$ given by
\begin{equation}
  M_{ij}=\beta(P(a_i),P(b_j)).
  \label{eq:M}
\end{equation}
If $i<j$, then the skew hypothesis gives
$\dist(a_i,b_j)\leq t=d-1$; hence Lemma~\ref{lem:zero-pattern} gives $M_{ij}=0$.
On the diagonal,
$\dist(a_i,b_i)\geq t+1=d$, so $M_{ii}\neq 0$.
Thus $M$ is lower triangular with nonzero diagonal, and therefore
\begin{equation}
  \operatorname{rank}_K M=m.
  \label{eq:rank-lower}
\end{equation}

To obtain the complementary upper bound, order the basis
$(e_S)_{S\subseteq[d]}$ of $\mathcal{A}$.
Let $R$ be the $m\times 2^d$ matrix whose $i$th row is the coordinate vector of $P(a_i)$, and let $C$ be the analogous matrix for the vectors $P(b_j)$.
Let $H$ be the $2^d\times 2^d$ matrix of $\beta$ in this basis.
Bilinearity gives
\begin{equation}
  M=RHC^{\mathsf T}.
  \label{eq:factorization}
\end{equation}
Consequently,
\[
  \operatorname{rank}_K M\leq 2^d=2^{t+1}.
\]
Together with \eqref{eq:rank-lower}, this proves $m\leq 2^{t+1}$. The proof is complete.
\end{proof}

\begin{remark}
Lemma~\ref{lem:zero-pattern} converts the one-sided distance assumptions into a triangular zero pattern, while the $2^d$-dimensional algebra limits the rank of every matrix obtained from the kernel.
Neither the word length $n$ nor the size of the ambient alphabet enters the final bound; only the finitely many symbols appearing in the family must be encoded.
\end{remark}

\section*{Declaration on the use of generative AI}

The authors used generative AI tools to assist in discussing proof strategies, checking proofs, and
improving the exposition. The authors take full responsibility for the mathematical arguments, results,
and conclusions, all of which were carefully reviewed and verified by them.

\end{document}